\documentclass{article}

\RequirePackage{amsthm,amsmath,amsfonts,amssymb}
\RequirePackage[authoryear]{natbib}
\RequirePackage[colorlinks,citecolor=blue,urlcolor=blue,unicode=true,linkcolor=blue]{hyperref}
\RequirePackage{hyperref}
\usepackage{enumitem}
\usepackage{authblk}

\theoremstyle{plain}
\newtheorem{theorem}{Theorem}
\newtheorem{lemma}{Lemma}
\newtheorem{proposition}{Proposition}
\newtheorem{corollary}{Corollary}
\theoremstyle{definition}

\newcommand{\xdom}{\mathcal{X}}
\newcommand{\R}{\mathbb{R}}
\newcommand{\N}{\mathbb{N}}
\newcommand{\E}{\mathbf{E}}
\newcommand{\Nl}{\mathcal{N}}
\newcommand{\x}{x}
\newcommand{\xt}[1][t]{{\x}_{#1}}

\newcommand{\vx}{X}
\newcommand{\vxt}[1][t]{\vx_{#1}}
\newcommand{\vxtt}{\vxt[t+1]}

\newcommand{\st}[1][t]{\sigma_{#1}}
\newcommand{\stt}{\st[t+1]}
\newcommand{\C}{C}
\newcommand{\Ct}[1][t]{\C_{#1}}
\newcommand{\Ctt}{\Ct[t+1]}
\newcommand{\Ctilde}[1][t]{\tilde{\C}_{#1}}

\newcommand{\Id}[1][\dim]{\mathbf{Id}_{#1}}

\newcommand{\pc}[1][t]{p_{#1}}
\newcommand{\pcc}{\pc[t+1]}
\newcommand{\ptarget}{p_{\textrm{target}}}
\newcommand{\fac}{\beta}

\renewcommand{\dim}{n}

\newcommand{\alg}[1][]{\mathcal{A}^{#1}}
\newcommand{\algt}{\alg[t]}
\newcommand{\dirac}[1]{\delta_{#1}}
\newcommand{\paramt}[1][]{\gamma_{#1}}

\newcommand{\vparamt}{{\boldsymbol{\paramt[]}}}
\newcommand{\vparamtt}[1][t]{\vparamt_{#1}}
\newcommand{\vparamttt}{\vparamtt[t+1]}
\newcommand{\paramk}[1][]{\theta_{#1}}
\newcommand{\vparamk}{{\boldsymbol{\paramk[]}}}
\newcommand{\vparamkt}[1][t]{\vparamk_{#1}}
\newcommand{\vparamktt}{\vparamkt[t+1]}
\newcommand{\paramkt}[1][t]{\paramk[#1]}

\newcommand{\paramu}[1][]{\psi_{#1}}
\newcommand{\paramut}[1][t]{\paramu[#1]}
\newcommand{\vparamu}{{\boldsymbol{\paramu[]}}}
\newcommand{\vparamut}[1][t]{{{\vparamu}_{#1}}}
\newcommand{\ptspace}{\Gamma}
\newcommand{\pkspace}{\Theta}
\newcommand{\puspace}{\Psi}

\newcommand{\xsalgebra}{\Sigma}
\newcommand{\ind}{\mathbf{1}}
\newcommand{\cspace}[1][\dim]{SD_{#1\times #1}}

\newcommand{\strong}[1]{\textit{#1}}

\newcommand{\qp}[1][]{q_{\paramk[#1]}}

\newcommand{\Qp}[1][]{Q_{\paramk[#1]}}
\newcommand{\Qpt}[1][t]{Q_{\vparamkt[#1]}}
\newcommand{\pacct}{\alpha_{\paramk}(x, y)}
\newcommand{\ksalg}{\Sigma_\pkspace}
\newcommand{\usalg}{\Sigma_\puspace}
\newcommand{\dx}{d_\xdom}
\newcommand{\dk}{d_\pkspace}
\newcommand{\xmeasure}{\lambda}
\newcommand{\B}[1][\cdot]{B_{\paramk{}}(x, #1)}
\newcommand{\Bt}[1][\cdot]{B^t_{\paramk{}}(x, #1)}
\newcommand{\ack}[1][\paramk]{\tilde{P}_{#1}}
\newcommand{\ackern}[1][\cdot]{\ack(x, #1)}
\newcommand{\ackernt}[1][\cdot]{\ack^t(x, #1)}
\newcommand{\acp}[1][\paramk]{\tilde{p}_{#1}}
\newcommand{\acpt}[1][\paramk]{\tilde{p}^t_{#1}}

\newcommand{\chaint}{(\vxt{}, \vparamkt{}, \vparamut{})_{t\in \N}}
\newcommand{\initchain}{(\xt[0], \paramkt[0], \paramut[0])}
\newcommand{\ttheta}{{\tilde{\theta}}}

\newcommand{\vyt}[1][t]{Y_{#1}}

\newcommand{\tgamma}[1][]{{\tilde{\gamma_{#1}}}}
\newcommand{\Cx}{C_\xdom}
\newcommand{\Ctheta}{C_\Theta}

\newcommand{\gpi}{g_\pi}
\newcommand{\xinit}{\nu_0}

\NewDocumentCommand{\Dct}{ O{t} O{\Cx}}{D_{#2,#1}}

\DeclareMathOperator{\Det}{Det}

\DeclareMathOperator{\cov}{cov}

\begin{document}

\title{Ergodicity of an Adaptive MCMC Sampler under a Probability Bound}
\author{Alexandre~Chotard \\ {\tt alexandre.chotard@univ-littoral.fr}}
\affil{OSMOSE, Université du Littoral Côte d'Opale}
\date{}

\maketitle

\begin{abstract}
This paper provides sufficient conditions over the sequence of samples and parameters of an adaptive Markov Chain Monte Carlo (MCMC) algorithm to ensure ergodicity with respect to a target distribution that can have unbounded support. These conditions aim to make more easily usable the conditions of Containment and Diminishing Adaptation from \cite{roberts2007ergodicitymcmc} formulated over the transition kernels, without needing, as was done in other works, an artificial assumption of the compactness over both sample and parameter spaces. The paper shows that the condition of compactness can be relaxed to a more realistic bound in probability over the sequence of both samples and parameters.
\end{abstract}

\section{Introduction}

Markov Chain Monte Carlo (MCMC) methods are a widely used family of methods which aim to generate samples from a target distribution $\pi$ by constructing a Markov chain $(\vxt)_{t \in \N}$ with Markov kernel $P$ such that asymptotically, $\vxt$ will have the same law as $\pi$. A good choice of the kernel $P$ relatively to $\pi$ is essential for the efficiency of the sampling process, and different techniques allow for online adaptation of a parametrized kernel $P_{\paramt}$ for $\paramt \in \ptspace$ in order to improve sampling efficiency (see e.g. \cite{roberts2009examples}). However, it is well known that even if for any $\paramt \in \ptspace$ each kernel $P_{\paramt}$ were ergodic, the adaptive process itself may not be ergodic relatively to $\pi$ (see e.g.~\cite{roberts2007ergodicitymcmc}).

A large body of work has been developed in order to provide conditions to ensure the ergodicity of an adaptive MCMC, with different approaches: using mixingales~\cite{haario2001,atchade2005adaptivemcmc}, martingale decomposition solving the Poisson equation 
\cite{andrieu2006ergadaptivemcmc,saksman2010ergadaptivemcmcunbounded}, a coupling construction~\cite{roberts2007ergodicitymcmc} and others~\cite{atchade2010adaptivemcmcresolvents,chimisov2018airmcmc}. This paper focuses on the coupling approach~\cite{roberts2007ergodicitymcmc}, which guarantees the ergodicity of the adaptive process provided that the Markov chain is in some sense contained, and that the adaptation diminishes to $0$ over time. Different variations to the conditions of containment and of diminishing adaptation exist, and here we follow the formulation of \cite[Theorem~2]{roberts2007ergodicitymcmc}: 
\begin{enumerate}[label=(A\arabic*)]
	\item  \label{cond:probability_bound} Containment: for all $\epsilon>0$, let $M_\epsilon$ denote the function
	\begin{equation*} \label{eq:Mepsilon}
		M_\epsilon : (x, \paramt) \mapsto \inf_{t \in \N^*} \lbrace \|P^t_{\paramt}(x, \cdot) - \pi\| < \epsilon\rbrace \enspace .
	\end{equation*}
	For any initial condition $(\xt[0], \paramt[0])$ and $\epsilon >0$ the sequence $(M_\epsilon(\vxt, \vparamtt))_{t \in \N}$ is bounded in probability, i.e. for all $\delta > 0$ there exists $T_\epsilon > 0$ such that $\Pr(M_{\epsilon}(\vxt, \vparamtt) < T_\epsilon| \vxt[0]=\xt[0], \vparamtt[0]=\paramt[0]) \geq 1 - \delta$ for all integer $t$.
\end{enumerate}
\begin{enumerate}[label=(A2.a)]
	\item  \label{cond:diminishing_adaptation} Diminishing Adaptation: the random variable $D_t := \sup_{x \in \xdom} \| P_{\vparamttt}(x, \cdot) - P_{\vparamtt}(x, \cdot)\|$ converges in probability to $0$.
\end{enumerate}
If $\xdom \times \ptspace$ is compact, then the continuity of the densities underlying the adaptive MCMC can be used to easily show that \ref{cond:probability_bound} and \ref{cond:diminishing_adaptation} hold (see e.g. \cite[Corollary~4]{roberts2007ergodicitymcmc}). On unbounded domains however, these conditions are more difficult to show: \cite{craiu2015stability} proved the ergodicity of the adaptive process, with the additional assumptions that the kernels have bounded jumps, and that outside of some compact $K$, the process switches to a fixed kernel. This approach has been extended in \cite{rosenthal2018ergcombocontinuous} to truncated combocontinuous densities. 

In the context of adaptive MCMC with continuous densities, this paper aims to make the conditions of containment and diminishing adaptation easier to verify. First, a minor change is that the parameter  $\paramt$ is split into $(\paramk, \paramu)$ in $\pkspace \times \puspace = \ptspace$, where $\paramk$ parametrizes the sampling distribution $P_{\paramk}$, and $\paramu$ contains parameters that controls the adaptation process itself. This modelling allows a little more leeway: while the sequence of distribution parameters $(\vparamkt)_t$ will be required to have some stability, there is no direct constraint to the evolution of the adaptation variables $(\vparamut)_t$, as long as the sequences $(\vxt)_t$ and $(\vparamkt)_t$ remain well-behaved. For example, $\paramu$ could contain a parameter that reduces the size of the updates of $\paramk$, and this parameter would be free to diverge to infinity. Secondly, the conditions of containment and diminishing adaptation are expressed directly using topology and metrics on the state and parameter space, instead of total variation between kernels. With $\dk$ a metric on the space $\pkspace$:
\begin{enumerate}[label={(B\arabic*)}]
	\item \label{cond:containment_param} The sequences $(\vxt)_t$ and $(\vparamkt)_t$ are bounded in probability, i.e. for all $\delta > 0$, there exists a compact set $C_\delta$ on the product space $\xdom \times \pkspace$ such that $\Pr((\vxt, \vparamkt) \in C_\delta) \geq 1 - \delta$ for all $t \in \N$.
	\item \label{cond:da_param} The random variable $\dk(\vparamktt, \vparamkt)$ converges to $0$ in probability.
\end{enumerate}
Thirdly, ergodicity is shown without further assumption of compactness of the state space or the parameter space, which can be seen as a generalisation of \cite[Corollary~3]{roberts2007ergodicitymcmc}. The regularity provided by the continuity of functions underlying the kernels $(P_{\paramk})_{\paramk}$ suffices to show a sort of ``local tightness'' which is heavily exploited in this paper. 
Finally, this work also uses a slightly weaker form for diminishing adaptation, which suffice for our purpose:
\begin{enumerate}[label=(A2.b)]
	\item  \label{cond:diminishing_adaptation_Cx} for any compact set $\Cx \subset \xdom$, the random variable $\Dct := \sup_{x \in \Cx} \| P_{\vparamktt}(x, \cdot) - P_{\vparamkt}(x, \cdot)\|$ converges in probability to $0$.
\end{enumerate}

Note that proofs of ergodicity in this paper are conditional to \ref{cond:containment_param} and \ref{cond:da_param}: complete proofs of ergodicity still require the verification of these conditions, which is a difficult problem in itself. The difficulty ``solved'' here is mostly of tackling the unboundedness of the state and parameter spaces. On compact spaces, containment and diminishing adaptation follows trivially from \ref{cond:containment_param} and \ref{cond:da_param} with applications of Lebesgue's dominated convergence theorem. An additional, perhaps obvious note: compactness may seem to be a cheap assumption to make, as the compacts could be taken large enough to be irrelevant in practice. The price however, is that ergodicity becomes just as irrelevant: if an adaptive MCMC is only ergodic due to being constrained to a bound that it never reaches in practice, then the algorithm is useless and showing its ergodicity is misleading.

The paper is organized as follows: section~\ref{sec:preliminaries} introduces more thoroughly the setting and the notations used throughout the paper; section~\ref{sec:main} contains the remaining conditions used in the paper, the intermediary steps and the main results; finally section~\ref{sec:appli} applies the tools developed in this paper to two different adaptive Metropolis-Hastings algorithms. The appendix contains a short technical lemma used in section~\ref{sec:main}.

\section{Setting and Notations} \label{sec:preliminaries}

\subsection{Notations}

The set of non-negative integers and of real numbers are respectively denoted $\N$ and $\R$.  The size $\dim$ identity matrix is denoted $\Id$, and $\cspace$ for the space of $\dim\times\dim$-positive definite symmetric matrices with coefficients in $\R$. For $A$ a matrix and  $x$ and $y$ vectors of $\R^n$, $A^T$ is the transpose of $A$ and $xy^T$ denotes the outer product between $x$ and $y$. For $X$ a random variable, $X \sim \pi$ denotes that $X$ follows a distribution $\pi$, and $\E(X)$ is the expected value of $X$. For $x \in \R^n$ and $C \in \cspace$, $\Nl(x, C)$ denotes the multivariate normal distribution with mean $x$ and covariance matrix $C$. Elements of $\xdom$, $\pkspace$ and $\puspace$ are typically denoted $x$, $\paramk$ and $\paramu$, while random variables over these spaces are denoted $\vxt[]$, $\vparamk$ and $\vparamu$.

\subsection{Setting}

The state space $\xdom$ is assumed to be a metric set equipped with a distance $\dx$, a sigma-algebra $\xsalgebra$ and a measure $\xmeasure$, on which is defined the target probability distribution $\pi$ that we want to sample from, that we suppose absolutely continuous relatively to $\xmeasure$, with density $\gpi$. The parameter and adaptation spaces $\pkspace$ and $\puspace$ are measurable spaces with sigma-algebras $\ksalg$ and $\usalg$, and $\pkspace$ admits a distance function $\dk$.

A MCMC algorithm with parameter $\paramk \in \pkspace$ and probability kernel $P_{\paramk}$ defines a time-homogeneous Markov-chain $(\vxt)_{t \in \N}$, which we want  \strong{ergodic} relatively to $\pi$, i.e. that for all $x \in \xdom$,
\begin{equation}	\label{eq:ergodicity}
	\lim_{t \to \infty} \| P^t_{\paramk}(x, \cdot)  - \pi \| = 0 \enspace ,
\end{equation} 
where $\|\nu\|$ denotes the total variation of a signed measure $\nu$ on $\xsalgebra$, that is
\begin{equation}	\label{eq:totalvariation}
	\|\nu\| := \sup_{A \in \xsalgebra} |\nu(A)| \enspace .
\end{equation}

A Metropolis-Hastings algorithm (MH) uses a so-called proposal or candidate distribution $\Qp$ to sample from some $\vxt$ valued in $\xdom$ some candidate $\vyt$ also in $\xdom$. Suppose that for all $x \in \xdom$, $\Qp(x, \cdot)$ is absolutely continuous with respect to $\xmeasure$, with density $\qp(x, \cdot)$, and let $\pacct$ denote the probability of accepting $y$ knowing $x$, defined as
\begin{equation}
	\pacct := \min\left(1,  \frac{\gpi(y) \qp(y,x)}{\gpi(x)\qp(x,y)} \right) \enspace .
\end{equation}
From an initial point $\vxt[0]$ following some initial distribution $\xinit$, the MH algorithm accepts the candidate $\vyt \sim \Qp(\vxt, \cdot)$ with probability $\alpha(\vxt, \vyt)$, and otherwise doesn't move:
\begin{equation}
	\vxtt := \left\lbrace \begin{array}{l}
		\vyt \textrm{ with probability } \alpha_{\paramk}(\vxt, \vyt) \enspace , \\
		\vxt \textrm{ else.}
	\end{array} \right.
\end{equation}
The MH is a MCMC algorithm whose transition kernel $P_{\paramk}$ writes
\begin{equation} \label{eq:ptheta}
	P_{\theta}(x, dy) = \alpha_\theta(x, y) \Qp(x, dy) + (1 - a(x, \theta)) \dirac{x}(dy) \enspace ,
\end{equation}
where $\dirac{x}$ denotes the Dirac measure at $x$, and $a(x, \theta)$ is the probability of a candidate being accepted, i.e.
\begin{equation} \label{eq:praccept}
	a(x, \theta) := \int_\xdom \alpha_\theta(y | x) \Qp(dy | x) \enspace .
\end{equation}
Note that when $\qp$ is symmetric with respect to $x$ and $y$, then $\pacct$ simplifies to $\min(1, \gpi(y)/\gpi(x))$.

An \strong{adaptive} MCMC defines a time-homogeneous Markov chain $\chaint$ on $\xdom \times \pkspace \times \puspace$. We denote $\alg$ the marginal of the chain's transition kernel with respect to $\xdom$, i.e. the function which to $(x, \paramk, \paramu, A) \in \xdom \times \pkspace \times \puspace \times \xsalgebra$ maps
\begin{equation}	\label{eq:alg}
	\algt(x, \paramk, \paramu, A) := \Pr\left( \vxt \in A | \vxt[0] = x, ~\vparamkt[0] = \paramk, ~\vparamut[0] = \paramu\right) \enspace .
\end{equation}
An adaptive MCMC or $\alg$ is called \strong{ergodic} with respect to $\pi$ if for any initial condition $\initchain$, $\|\algt(x_0, \paramk[0], \paramu[0], \cdot) - \pi\|$ converges to $0$. 
As previously mentioned, even if each non-adaptive kernel $P_{\paramk}$ is ergodic for any ${\paramk} \in \pkspace$, successive modifications of ${\paramk}$  can prevent the convergence of $\algt$ to $\pi$: an intuitive example of this phenomenon is a MH algorithm where the scaling of the proposal is decreased outside of a region and increased inside the same region, thus creating a bias to be outside of this region; or if its proposal $\Qpt$ would diverge. 

\section{Main Results}	\label{sec:main}

\subsection{Conditions}	\label{subsec:cond}

I first present the conditions needed in the paper. In addition to conditions \ref{cond:probability_bound}, \ref{cond:diminishing_adaptation}, \ref{cond:diminishing_adaptation_Cx}, \ref{cond:containment_param} and \ref{cond:da_param} previously introduced, we require the ergodicity of the non-adaptive kernels:
\begin{enumerate}[label={(B\arabic*)}]
	\setcounter{enumi}{2}
	\item \label{cond:pergodic} For every $\paramk{} \in \pkspace$, the transition kernel $P_{\paramk}$ is ergodic relatively to $\pi$.
\end{enumerate}
The next conditions focus on the continuity of parts of the kernel $P_{\paramk}(x, \cdot)$ over both $x$ and $\paramk$. Note that, in many applications of interest such as for Metropolis-Hastings algorithms, the kernel itself is not continuous, in the sense that the total variation $\|P_{\paramk}(x,\cdot) - P_{\tilde{\paramk}}(y,\cdot)\|$ does not go to $0$ when $(x, \paramk)$ goes to $(y, \tilde{\paramk})$. Indeed, for a Metropolis-Hastings, the difference between the two Dirac contained in the singular parts of the kernels remain constant for $x \neq y$. Therefore, following Lebesgue's decomposition Theorem with respect to the measure $\xmeasure$, the kernel $P^t_{\paramk}(x, \cdot)$ is split into its singular part $\Bt$ and absolutely continuous part $\ackernt$ with density function $\acpt(x, \cdot)$. The condition of continuity applies to the density and the total mass of each part, instead of the full kernel:
\begin{enumerate}[label={(C\arabic*)}]
	\item For all integer $t\geq 1$, the functions $\tilde{p}^t : (x, \paramk{}, y) \mapsto \acpt(x,y)$ and $(x, \theta) \mapsto \ackernt[\xdom]$ are continuous.
	\label{cond:ptcontinuity}
\end{enumerate}
This of course implies the continuity of $\Bt[\xdom]$, the mass of the singular part, and as shown later in lemma~\ref{lemma:tvc0}, is sufficient to ensure the continuity of $\|P^t_\theta(x, \cdot) - \pi\|$, and connect, as done in proposition~\ref{proposition:containment}, the condition for containment \ref{cond:probability_bound} to \ref{cond:containment_param}.  
In order to verify diminishing adaptation, it is useful to have the continuity of $\|P_{\theta}(x,\cdot) - P_{\ttheta}(x,\cdot)\|$ (note that only $\theta$ and $\ttheta$ differ). As shown in proposition~\ref{prop:dacompact}, this can be ensured with this additional assumption on the singular part $\B$, connecting our weaker condition for diminishing adaptation \ref{cond:diminishing_adaptation_Cx} to \ref{cond:da_param}:
\begin{enumerate}[resume,label={(C\arabic*)}]
	\item The family $(\B)_{(x, \theta) \in \xdom \times \Theta}$ is a family of discrete distributions tied in the following way: there exists a countable set $I$ such that for any $(x,\theta)$, $\B$ can be written as a sum of dirac at points $s_i(x)$ with weights $b_i(x,\theta)$, such that $\B = \sum_{i \in I} b_i(x,\theta) \delta_{s_i(x)}$. Furthermore, the functions $(b_i)_i$ and $(s_i)_i$ are continuous with respect to both $x$ and $\theta$. \label{cond:bdiscrete}
\end{enumerate}
Note that the functions $s_i$ do not depend of $\theta$: thus this condition restrict, crucially, $B_\theta(x, \cdot)$ and $B_\ttheta(x, \cdot)$ to share the same points (albeit with possibly different, but continuous, weights).

The remaining conditions are not strictly needed, but can accelerate proofs, and are provided for convenience.
In particular, verifying condition~\ref{cond:ptcontinuity} for all $t\in \N$ can be needlessly tedious, and can be done more easily (see  lemma~\ref{lemma:ptc0}), provided that the following conditions hold:
\begin{enumerate}[label={(C\arabic*)},resume]
	\item The function mapping $(x, \paramk{}, y) \to \acp(x, y)$ is jointly continuous, and the function $(x, \paramk{}) \to \ackern[\xdom]$ is jointly continuous. \label{cond:pcontinuity}
	\item The family $(p_\theta(x, y))_{x \in \xdom}$ is locally bounded, i.e. for any $(\theta_0, y_0) \in \Theta \times \xdom$ there exists $U$ a neighbourhood of $(\theta_0, y_0)$ and  $M$ such that $p_\theta(x,y) < M$ for all $x \in \xdom$ and all $(\theta,y) \in U$. \label{cond:plocalbound}
	\item Condition~\ref{cond:bdiscrete} hold, and for any $s_i$, either $s_i$ has a constant value $a_i \in \xdom$, either $s_i$ is the identity function $s_i(x) = x$.
	\label{cond:bsupcond} 
\end{enumerate} 
Condition~\ref{cond:bsupcond} was formulated to include Metropolis-Hastings methods (with a unique atom $s_0(x) = x$ and weight $b_0$ equal to the rejection probability $1-a(x,\theta)$, see \eqref{eq:ptheta}), and the possibility of ``jumping'' to some fixed point. Different and more general assumptions on the functions $(s_i)_i$ would also work.
For an adaptive Metropolis-Hastings, as condition~\ref{cond:bsupcond} is verified, instead of \ref{cond:pcontinuity} and \ref{cond:plocalbound}, the following condition suffice:
\begin{enumerate}[resume,label={(C\arabic*)}]	
	\item the family of measures $(\Qp(x, \cdot))_{x,\paramk}$ is absolutely continuous with respect to $\xmeasure$, admitting densities $(\qp(x,\cdot))_{x,\paramk}$. Furthermore, the density $\gpi$ and the function $q:(x, y, \paramk) \mapsto \qp(x,y)$ are continuous functions, and the family $(\qp(x,y))_{x \in \xdom}$ is locally bounded (see condition \ref{cond:plocalbound}).	\label{cond:qcontinuity}
\end{enumerate}

Note that there is no direct condition on the sequence $(\vparamut)_t$, only indirectly that conditions~\ref{cond:containment_param} and \ref{cond:da_param} hold. 
The following subsection details intermediary results which are later used to transfer the continuity of $p$, $b_i$ and $s_i$ from condition~\ref{cond:ptcontinuity} and condition~\ref{cond:bdiscrete} to continuity in total variation.

\subsection{Continuity Results}	\label{subsec:continuity}

When $\xdom$ is compact, the continuity of $\|P^t_\theta(x,\cdot) - \pi\|$ can easily be derived from condition~\ref{cond:ptcontinuity} using Lebesgue's dominated convergence theorem. 
An important property that allows us to fall back to this simple case is the property of tightness: as a reminder, for $F$ a family of measures over a Hausdorff space $(\xdom, T)$ equipped with a $\sigma$-algebra $\xsalgebra$ which contains the topology $T$, the family $F$ is called \strong{tight} if for any $\epsilon > 0$ there exists a compact subset $K_\epsilon$ of $\xdom$ such that $\mu(K_\epsilon^c) < \epsilon$ for all $\mu \in F$. Under condition~\ref{cond:pcontinuity}, there is a ``local tightness'' in the sense that for any compact $C$ of $\xdom\times \Theta$, the family $(\tilde{P}_\theta(x,\cdot))_{(x,\theta) \in C}$ is tight. This is expressed in a more general context in the following proposition.

\begin{proposition}	\label{prop:tightness}
	Let $(\xdom, \xsalgebra)$ be a $\sigma$-locally compact Hausdorff measurable space equipped with a locally finite measure $\xmeasure$, 
	and let $\Gamma$ be a locally compact Hausdorff space. Let $(\mu_{\gamma})_{\gamma \in \Gamma}$ be a family of finite measures absolutely continuous with respect to $\xmeasure$, with densities $(h_{\gamma})_\gamma$. If the function $h: (x,\gamma) \mapsto h_{\gamma}(x)$ and the total mass function $m: \gamma \mapsto \mu_{\gamma}(\xdom)$ are continuous with respect to $\gamma$, then for any compact set $C$ of $\Gamma$, the family of measures $(\mu_\gamma)_{\gamma \in C}$ is tight.
\end{proposition}

\begin{proof}
	Let $C$ be a compact subset of $\Gamma$. As $\xdom$ is $\sigma$-locally compact, there exists a nested sequence of  compact sets $(K_i)_{i \in \N}$ with $K_i \subset K_{i+1}$ such that $\cup_i K_i = \xdom$. Let $m_i$ denote the function $\gamma \in C \mapsto \mu_\gamma(K_i)$ and $m_C: \gamma \in C \mapsto \mu_\gamma(\xdom)$ the restriction of $m$ to $C$. If the sequence $(m_i)_i$ converges uniformly to $m_C$, then the family $(\mu_\gamma)_{\gamma \in C}$ is tight: indeed, then for all $\epsilon > 0$, there exists $T_\epsilon$ such that for all $i \geq T_\epsilon$, $|m_i(\gamma) - m_C(\gamma)| < \epsilon$ for all $\gamma \in C$. As $|m_i(\gamma) - m_C(\gamma)|$ equals $\mu_\gamma(K_i^c)$, we have $\mu_\gamma(K_{T_\epsilon}^c) < \epsilon$ for any $\epsilon> 0$ and any $\gamma \in C$, i.e. $(\mu_\gamma)_{\gamma \in C}$ is tight.
	
	The uniform convergence of $(m_i)_i$ to $m_C$ follows from Dini's theorem. Dini's theorem may be applied as $m_C$ is continuous (by assumption), that $(m_i)_i$ and $m_C$ are defined on a compact set $C$, that $(m_i)_i$ is a monotonically increasing sequence which converges pointwise to $m_C$ (by the monotone convergence theorem), and that each $m_i$ is a continuous function. This last point follows from the continuity of $h_\gamma(x)$ with respect to $\gamma$, which is thus locally bounded on $K_i$ by some $M_i$, therefore integrable on $K_i$ (since $\xmeasure(K_i) < \infty$, as the local finiteness of $\xmeasure$ implies finiteness on compacts). Thus Lebesgue's dominated convergence theorem can be applied, showing the continuity of $m_i$.
\end{proof}

This ``local tightness'' can be used to easily show continuity under integration, as in the following lemma.
\begin{corollary}  \label{coro:tightc0}
	Let $(\xdom, \xsalgebra)$ be a $\sigma$-locally compact Hausdorff measurable space equipped with a locally finite measure $\xmeasure$, and let $\Gamma$ be a locally compact Hausdorff space. 
	Consider $(\mu_\gamma)_{\gamma \in \Gamma}$ a family of finite measure absolutely continuous with respect to $\xmeasure$, with densities $(h_{\gamma})_\gamma$. Suppose that the function $h : (x, \gamma) \mapsto h_\gamma(x)$ is continuous with respect to $\gamma$, and that the family $(\mu_\gamma)_{\gamma \in C}$ is tight for any $C$ compact subset of $\Gamma$. For any function $f:\xdom \times \Gamma \to \R$ continuous with respect to $\gamma$ and such that the family $(f(x,\gamma))_{x \in \xdom}$ is locally bounded, the function $\mu f: \gamma \mapsto \int f(x, \gamma)h_\gamma(x)\xmeasure(dx)$ and the function $\Delta \mu f: (\gamma, \tgamma) \mapsto \int |f(x, \gamma)h_\gamma(x) - f(x,\tgamma)h_{\tgamma}(x)|\xmeasure(dx)$ are continuous.
\end{corollary}

\begin{proof}
	This is easily derived from Lebesgue's dominated convergence theorem. Take $\gamma_0 \in \Gamma$, $M>0$ and $C$ a compact neighbourhood of $\gamma_0$ such that $|f(x,\gamma)| < M$ for all $x\in\xdom$ and $\gamma\in C$. As $(\mu_\gamma)_{\gamma \in C}$ is tight (by proposition~\ref{prop:tightness}), for any $\epsilon>0$ there exists $K_\epsilon$ a compact subset of $\xdom$ such that $\mu_\gamma(K_\epsilon) < \epsilon/M$ for all $\gamma \in C$. Thus for $\gamma \in C$, $\mu f (\gamma)$ can be divided into $\int_{K_\epsilon}f(x,\gamma)h_\gamma(x)\xmeasure(dx)$, which by Lebesgue's dominated convergence theorem is a continuous function of $\gamma$, and $\int_{K_\epsilon^c}f(x,\gamma)h_\gamma(x)\xmeasure(dx)$ which is inferior in absolute value to $\epsilon$. Consequently, the difference $|\mu f(\gamma) - \mu f(\tgamma)|$ can be made arbitrarily small by taking $\tgamma$ close enough to $\gamma$, meaning $\mu f$ is continuous. The same reasoning can be applied to $\Delta \mu f$.
\end{proof}

Proposition~\ref{prop:tightness} and corollary~\ref{coro:tightc0} can also be used on families of discrete measures, as in the following corollary.

\begin{corollary}	\label{coro:tightbc0}
	Let $(\xdom, \xsalgebra, \xmeasure)$ be an Hausdorff measure space, and let $\Gamma$ be a locally compact Hausdorff space. Consider $(B_\gamma)_{\gamma \in \Gamma}$ a family of finite discrete measure with respect to $\xmeasure$, and suppose that this family can be written as the sum of weighted Dirac $\sum_{i\in\N} b_i(\gamma) \dirac{s_i(\gamma)}$. Suppose that $b_i(\gamma)$ and $B_\gamma(\xdom)$ are continuous functions of $\gamma$ for any $i$, then for any compact set $C$ of $\Gamma$ and any $\epsilon > 0$, there exists $T_\epsilon$ such that $\sum_{i\geq T_\epsilon} b_i(\gamma) < \epsilon$ for all $\gamma \in C$. 
	
	Furthermore, for any locally bounded family of continuous functions $(f_i)_i$ from $\xdom \times \Gamma$ to $\R$, the function $Bf : (x,\gamma) \mapsto \sum_i b_i(\gamma)f_i(x, \gamma)$ and the function $\Delta B f: (x, \gamma, \tgamma) \mapsto \sum_i |b_i(\gamma)f_i(x, \gamma) - b_i(\tgamma)f_i(x,\tgamma)|$ are continuous with respect to $x$ and $\gamma$. 
\end{corollary}

\begin{proof}
	This is a direct application of proposition~\ref{prop:tightness} and corollary~\ref{coro:tightc0}: let $\mu_\gamma$ denote the measure $I \subset \N \mapsto \sum_{i \in I} b_i(\gamma)$. By proposition~\ref{prop:tightness}, taking $\xmeasure$ in the proposition as the counting measure, for any compact $C$ of $\Gamma$, the family $(\mu_\gamma)_{\gamma \in C}$ is tight. The  corollary follows, using the same reasoning as in  corollary~\ref{coro:tightc0} (and using continuity over both $x$ and $\gamma$).
\end{proof}

We can now show the continuity of the total variation $\|P_\theta^t(x, \cdot) - \pi\|$ with respect to $x$ and $\paramk$.

\begin{lemma} \label{lemma:tvc0}
	Consider an adaptive MCMC algorithm targeting distribution $\pi$, for which condition~\ref{cond:ptcontinuity} holds. Then for any $t \geq 1$, the total variation $(x, \theta) \mapsto
	\|P^t_{\paramk}(x, \cdot)  - \pi \|$ is jointly continuous in both $x$ and $\paramk$.
\end{lemma}

\begin{proof}
	As $\pi$ is absolutely continuous with respect to $\xmeasure$, the total variation $\|P^t_\theta(x, \cdot)  - \pi \|$ can be written as
	\begin{align*}
		\|P^t_\theta(x, \cdot)  - \pi \| &= \frac{1}{2} \left( B^t_\theta(x, \xdom)  + \int |\tilde{p}^t_\theta(x, y) - \gpi(y)|\xmeasure(dy)  \right) \enspace .
	\end{align*}
	By condition~\ref{cond:ptcontinuity}, $B^t_\theta(x, \xdom)$ which equals $1 - \tilde{P}^t_\theta(x, \xdom)$ and $\tilde{p}^t_\theta(x, y)$ are jointly continuous. By proposition~\ref{prop:tightness}, for any compact $C$ of $\xdom \times \Theta$ the family $(\tilde{P}_\theta(x, \cdot))_{(x,\theta)\in C}$ is tight, so for any $\epsilon>0$ we can find a compact $K_\epsilon$ of $\xdom$ for which $\tilde{P}_\theta(x, K_\epsilon^c) < \epsilon/2$ for all $(x,\theta) \in C$. The compact $K_\epsilon$ can be taken large enough to ensure that $\pi(K_\epsilon^c) < \epsilon/2$ also. Thus, as in corollary~\ref{coro:tightc0}, $\int |\tilde{p}^t_\theta(x, \cdot) - \gpi|d\xmeasure $ can be divided with $K_\epsilon$ into a continuous function and a quantity lower than $\epsilon$ for any $(x,\theta) \in C$, implying continuity.	
\end{proof}

Similarly, for any compact $\Cx$ of $\xdom$ the $\sup_{x \in \Cx} \|P_\theta(x, \cdot) - P_\ttheta(x, \cdot)\|$ can be shown to be a continuous function of $\theta$ and $\ttheta$.
\begin{lemma} \label{lemma:deltac0}
	Suppose that condition~\ref{cond:pcontinuity} holds. Then the function $\Delta \tilde{P} : (x, \tilde{x}, \theta, \ttheta ) \mapsto \int |\acp(x, y) - \acp[\ttheta](\tilde{x}, y)|\xmeasure(dy)$ is continuous. If additionally, condition~\ref{cond:bdiscrete} holds, then the function $\Delta B : (x, \tilde{x}, \theta, \ttheta) \mapsto \sum_i | b_i(x,\theta) - b_i(\tilde{x}, \ttheta)|$ is continuous. Furthermore, for any compact subset $\Cx$ of $\xdom$, the function $V_{\Cx}: (\theta, \ttheta )  \mapsto \sup_{x \in \Cx} \|P_\theta(x, \cdot) - P_\ttheta(x, \cdot)\|$ is continuous.
\end{lemma}

\begin{proof}
	The continuity of $\Delta \tilde{P}$ and $\Delta B$ follow from corollary~\ref{coro:tightc0} and corollary~\ref{coro:tightbc0} (taking $f \equiv 1$). 
	As $s_i$ does not depend on $\theta$, the mass of $B_\theta(x, \cdot)$ and $B_\ttheta(x, \cdot)$ is on the same points and consequently
	\begin{align*}
		\|P_\theta(x, \cdot) - P_\ttheta(x, \cdot)\| = \frac{1}{2}\left(  \Delta B(x, x, \theta, \ttheta) + \Delta \tilde{P}(x,x, \theta, \ttheta) \right) \enspace ,
	\end{align*}
	which is therefore continuous. The continuity of $V_{\Cx}$ follows as taking the $\sup$ over a compact preserves continuity (see lemma~6 from the appendix).
\end{proof}

The next proposition is useful to avoid showing the continuity of $\acpt(x,y)$ for all $t > 1$. 

\begin{lemma} \label{lemma:ptc0}
	Suppose that conditions \ref{cond:pcontinuity}, \ref{cond:plocalbound} and \ref{cond:bsupcond} hold, then condition~\ref{cond:ptcontinuity} holds.
\end{lemma}

\newcommand{\ktack}{\tilde{P}_\theta^t\tilde{P}_\theta(x, \cdot)}
\begin{proof}
	Let $E_t$ be the property that the function $\tilde{p}^t: \theta, x, y \mapsto \acpt(x,y)$ and the function $(x,\theta) \mapsto \tilde{P}^t_\theta(x, \xdom)$ are continuous, and that $B^t_\theta(x,dy)$ can be written as a countable sum of Dirac distributions with points $s^t_i(x)$ and non-negative weights $b^t_i(x,\theta)$, $\sum_{i\in \N} b_i^t(x,\theta)\dirac{s^t_i(x)}(dy)$, such that for any $i\in\N$, $b_i^t$ and $s^t_i$ are continuous functions. By hypothesis, $E_1$ is verified. Suppose that $E_t$ holds for some $t \geq 1$. 
	Developing $P_\theta^{t+1}$,
	\begin{align*}
		P_\theta^{t+1}(x, dy) &= \int P_\theta^{t}(x, dz) P_\theta(z, dy) \\
		&= \tilde{P}_\theta^t \tilde{P}_\theta(x, dy) + \tilde{P}_\theta^t B_\theta(x, dy) + B_\theta^t \tilde{P}_\theta(x, dy) + B_\theta^t B_\theta(x, dy) \enspace , 
	\end{align*}
	where
	\begin{align*}
		\tilde{P}_\theta^t \tilde{P}_\theta(x, dy) &= \int \acpt(x, z) \acp(z,y) \xmeasure(dz)\xmeasure(dy) \enspace , \\
		\tilde{P}_\theta^t B_\theta(x, dy) &= \int \acpt(x, z) \sum_i b_i(z,\theta) \dirac{s_i(z)}(dy) \xmeasure(dz) \enspace , \\ 
		B_\theta^t \tilde{P}_\theta(x, dy) &= \int \sum_i b^t_i(x, \theta)\dirac{s^t_i(x)}(dz) \acp(z,y)\xmeasure(dy) \enspace , \\
		&= \sum_i b^t_i(x, \theta) \acp(s^t_i(x),y)\xmeasure(dy) \enspace ,\\
		B_\theta^t B_\theta(x, dy) &= \int \sum_i b^t_i(x, \theta)\dirac{s^t_i(x)}(dz) \sum_j b_j(z,\theta)\dirac{s_j(z)}(dy) \enspace , \\ 
		&= \sum_i \sum_j b^t_i(x,\theta) b_j(s^t_i(x),\theta) \dirac{s_j(s^t_i(x))}(dy) \enspace . 
	\end{align*}
	The measure $\ktack$ is absolutely continuous with respect to $\xmeasure$, with density $\bar{f}_1: (\theta,x,y): \int \acpt(x,z) \xmeasure(dz) \acp(z,y)$. By corollary~\ref{coro:tightc0}, the family of functions $(\acp(x,y))_{x \in \xdom}$ being locally bounded by condition~\ref{cond:plocalbound}, $\bar{f}_1$ is continuous. The measure $B_\theta^t \tilde{P}_\theta(x, \cdot)$ is also absolutely continuous with respect to $\xmeasure$ with density $\bar{f}_2 :(\theta, x, y) \mapsto \sum_i b^t_i(x,\theta)\acp(s^t_i(x), y)$, which is continuous by corollary~\ref{coro:tightbc0}. 
	The measure $B_\theta^T B_\theta(x, \cdot)$  is the sum of countable Dirac measures, with continuous weights and modes.
	Finally, $\tilde{P}_\theta^T B_\theta(x, \cdot)$ can admit singular or absolutely continuous parts, depending on the mode functions $(s_i)_i$. Let $I_1$ and $I_2$ be a partition of $\N$ for which $s_i(x) = x$ for all $i \in I_1$, and for all $i \in I_2$ there exists $a_i \in \xdom$ such that $s_i(x) = a_i$. Denoting $\nu_i(x,\theta)$ the measure $A \in \xsalgebra \mapsto \int \acpt(x,z)b_i(z,\theta) \xmeasure(dz) \dirac{s_i(z)}(A)$, for $i \in I_1$, $\nu_i(x,\theta)(A)$ equals $\int_A \acpt(x,z) b_i(z,\theta) \xmeasure(dz)$. and so is absolutely continuous with density $f_{3,i}: (\theta,x,y) = \acpt(x,y) b_i(y,\theta)$. The full measure on $I_1$ has density $\bar{f}_3 : (\theta,x,y) \mapsto \sum_{i \in I_1} \acpt(x,y) b_i(y, \theta)$, which is continuous by corollary~\ref{coro:tightbc0}.
	For $i \in I_2$, $\nu_i(x,\theta)(A) =  \int \tilde{p}^t_\theta(x,z) b_i(z,\theta) \xmeasure(dz) \ind_A(a_i)$, so is a Dirac in $a_i$ with weight $g_i(x,\theta) \mapsto \int \tilde{p}^t_\theta(x,z) b_i(z,\theta)\xmeasure(dz)$. As $b_i(z,\theta) \leq 1$ for all $i$, $z$ or $\theta$, we can apply corollary~\ref{coro:tightc0} to show that $g_i$ is a continuous function of $x$ and $\theta$. 
	
	Regrouping everything, $P^{t+1}_\theta(x,\cdot)$ can be decomposed into a countable sum of Dirac measures with continuous weights and modes
	\begin{align*}
		B^{t+1}_\theta(x, dy) = \sum_i \sum_j b^t_i(x,\theta) b_j(s^t_i(x),\theta) \dirac{s_j(s^t_i(x))}(dy) + \sum_{i \in I_2} g_i(x,\theta) \dirac{a_i}(dy) \enspace ,
	\end{align*}
	and an absolutely continuous part $\tilde{P}_\theta^{t+1}(x,\cdot)$ with continuous density $\bar{f}_1 + \bar{f}_2 + \bar{f}_3$. All that remains to show that $E_{t+1}$ holds is that the function $(x,\theta) \mapsto \tilde{P}_\theta^{t+1}(x,\xdom)$ is continuous. Let $F_i$ denote the function $(\theta, x) \mapsto \int \bar{f}_i(\theta, x, y) \xmeasure(dy)$. So 
	$$(F_1 + F_3)(x, \theta) = \int \acpt(x,z) \left(\tilde{P}_\theta(z, \xdom) + \sum_{i \in I_1} b_i(y, \theta)\right) \xmeasure(dz)\enspace .$$
	By hypothesis and by corollary~\ref{coro:tightbc0}, $\tilde{P}_\theta(x, \xdom)$ and $\sum_{i \in I_1} b_i(x, \theta)$ are continuous functions, and as they are bounded by $1$, the function $F_1 + F_3$ is continuous by corollary~\ref{coro:tightc0}. Similarly, $F_2$ is $\sum_i b_i^t(x, \theta) P_\theta(s_i(x), \xdom)$ and is continuous by corollary~\ref{coro:tightbc0}.
\end{proof}

Finally, a Metropolis-Hastings algorithms which fit condition~\ref{cond:qcontinuity} also satisfies our conditions for continuity, as shown in the following lemma.
\begin{lemma}	\label{lemma:mhc0}
	Consider an adaptive Metropolis-Hastings for which \ref{cond:qcontinuity} holds. Then conditions \ref{cond:ptcontinuity} and \ref{cond:bdiscrete} also hold.
\end{lemma}

\begin{proof}
	As detailed in Eq.~\eqref{eq:ptheta}, the kernel of a Metropolis-Hastings algorithm can be decomposed with respect to $\xmeasure$ into an absolutely continuous part with density $\tilde{p}_\theta(x,y)$ equal to $\alpha_\theta(x,y)q_\theta(x,y)$, and a singular part $(1-a(x,\theta))\dirac{x}$. By  lemma~\ref{lemma:ptc0}, it suffices to show that conditions \ref{cond:pcontinuity}, \ref{cond:plocalbound} and \ref{cond:bsupcond} hold: $\tilde{p}$ is continuous by continuity of $g_\pi$ and $q$;  $(\tilde{p}_\theta(x,y))_{x \in \xdom}$ is a locally bounded family as $(q_\theta(x,y))_{x \in \xdom}$ is a locally bounded family, and that $\alpha_\theta(x,y) \leq 1$; finally, by corollary~\ref{coro:tightc0}, $a(x, \theta) = \int \alpha_\theta(x, y)q_\theta(x,y) \xmeasure(dy)$ is continuous (using again that $\alpha_\theta(x,y) \leq 1$), showing the continuity of $\tilde{P}_\theta(x,\xdom)$.
\end{proof}

\subsection{Containment}

The following proposition is an adaptation of \cite[Corollary~3]{roberts2007ergodicitymcmc} to our context, and shows that under our continuity assumptions, containment on the state and parameter space implies containment on kernel space.

\begin{proposition}	\label{proposition:containment}
	Consider an adaptative MCMC algorithm such that condition~\ref{cond:pergodic} and condition~\ref{cond:ptcontinuity} hold. Then the adaptive MCMC is ``locally'' simultaneously ergodic on compact sets: for any compact $C$ of $\xdom \times \Theta$ and any $\epsilon > 0$, there exists $T$ such that for all $t \geq T$ and for all $(x,\theta) \in C$, $\|P^t_\theta(x, \cdot) - \pi\| < \epsilon$. Furthermore, if condition~\ref{cond:containment_param} also holds, then condition~\ref{cond:probability_bound} holds.
\end{proposition}

\begin{proof}
	The uniform ergodicity on compact sets is easily derived from lemma~\ref{lemma:tvc0} and using the same reasoning as in the proof of \cite[Corollary~3]{roberts2007ergodicitymcmc}. It is detailed here for the sake of completeness.
	For $\epsilon > 0$ and $t \in \N$, let $f_t$ denote the function $(x, \theta) \mapsto \|P^t_\theta(x, \cdot) - \pi\|$, and let $W_{t, \epsilon}$ be the set $\lbrace (x,\theta) \in \xdom \times \Theta | f_k(x,\theta) < \epsilon \textrm{ for all } k\geq t \rbrace$. As every kernel $P_\theta^t$ is ergodic with respect to $\pi$ (condition~\ref{cond:pergodic}),  $f_t(x,\theta)$ is a decreasing function of $t$, so $W_{t, \epsilon}$ is $f^{-1}_t((-\infty, \epsilon))$, which is an open set by lemma~\ref{lemma:tvc0}. 
	By condition~\ref{cond:pergodic} again, $\xdom \times \Theta$ is covered by the open sets $\cup_{t \in \N} W_{t, \epsilon}$, and thus for any compact $C$ subset of $\xdom \times \Theta$, a finite cover $(W_{t_i, \epsilon})_{i \in I}$ of $C$ can be extracted. Therefore, denoting $T := \max_{i \in I} t_i$, for any $(x,\theta) \in C$ and any $t \geq T$,  $\|P^t_\theta(x, \cdot) - \pi\| < \epsilon$.
	
	Under condition~\ref{cond:containment_param}, condition \ref{cond:probability_bound} follows: take $\delta>0$, and $C_\delta$ a compact subset of $\xdom \times \Theta$ such that $\Pr((\vxt,\vparamkt) \in C_\delta) > 1 - \delta$ for all $t \in T$. Then, as previously showed, there exists $T_\delta$ such that for all $t \geq T_\delta$ and $(x,\theta) \in C_\delta$, $\|P^t_\theta(x, \cdot) - \pi\| < \epsilon$. Let $E_{\epsilon, t, \delta}$ and $H_{t,\delta}$ be respectively the events $\lbrace M_\epsilon(\vxt, \vparamkt) > T_\delta \rbrace$ and $\lbrace (\vxt, \vparamkt) \in C_\delta \rbrace$.
	Thus 
	\begin{align*}
		Pr(M_{\epsilon}(\vxt, \vparamkt) \geq T_\delta) &= \Pr(E_{\epsilon, t, \delta} \cap H_{t,\delta}) + \Pr(E_{\epsilon, t, \delta} \cap H_{t,\delta}^c) \\
		&\leq 0 + \Pr(H_{t,\delta}^c) \\
		&< \delta \enspace .
	\end{align*}
\end{proof}

\subsection{Diminishing Adaptation}

When $(\vparamkt)_t$ is bounded in probability, lemma~\ref{lemma:deltac0} allows to translate diminishing adaptation in parameter space into diminishing adaptation in kernel space, as formalized in the following proposition.
\begin{proposition}	\label{prop:dacompact}
	Consider an adaptive MCMC algorithm for which the sequence $(\vparamkt)_t$ is bounded in probability. Suppose also that condition~\ref{cond:da_param}, condition~\ref{cond:bdiscrete} and condition~\ref{cond:pcontinuity} hold, then condition~\ref{cond:diminishing_adaptation_Cx} follows.
\end{proposition}

\begin{proof}
	Take $\epsilon>0$ and $\delta>0$, and let $V_{\Cx}$ denote the function $(\theta, \ttheta) \in \Theta^2 \mapsto \sup_{x \in \Cx} \|P_{\paramk}(x, \cdot) - P_{\ttheta}(x, \cdot)\|$. We need to show that for any initial condition $\initchain$ there exists $T$ such that, for all $t\geq T$, $\Pr(\Dct > \epsilon | \vxt[0]=\xt[0], \vparamkt[0]=\paramkt[0], \vparamut[0]=\paramut[0]) \leq \delta$, where $\Dct := V_{\Cx}(\vparamkt, \vparamktt)$. In the following, I will abuse notations and omit conditioning over the initial condition. As $(\vparamkt)_t$ is bounded in probability, there exists $\Ctheta$ a compact subset of $\Theta$ such that $\Pr(\vparamkt \in \Ctheta) > 1 - \delta/4$ for all $t$. Then for $H_t$ the event $\lbrace (\vparamkt, \vparamktt) \notin \Ctheta^2 \rbrace$, $\Pr(H_t) < \delta/2$.
	
	By lemma~\ref{lemma:deltac0}, $V_{\Cx}$ is continuous, so uniformly continuous on $\Ctheta^2$.  Hence, as $V_{\Cx}(\theta, \theta)$ equals $0$, there exists $\eta > 0$ such that if $(\theta, \ttheta) \in \Ctheta^2$ and if $\dk(\theta, \ttheta) < \eta$, then $V_{\Cx}(\theta, \ttheta) < \epsilon$. Let $E_{\eta,t}$ be the event $\lbrace \dk(\vparamkt, \vparamktt) < \eta \rbrace$. By condition~\ref{cond:da_param}, there exists $T_\delta$ such that for all $t\geq T_\delta$, $\Pr(E_{\eta,t}^c) < \delta/2$. For any $t \geq T_\delta$, the event $H_t^c \cap E_{\eta,t}$ is disjoint with the event $\lbrace \Dct > \epsilon \rbrace$, so
	\begin{align*}
		\Pr(\Dct > \epsilon) &= \Pr(\Dct > \epsilon \cap H_t) + \Pr(\Dct > \epsilon \cap H_t^c) \\
		&\leq \Pr(H_t) + \Pr(\Dct > \epsilon \cap H_t^c \cap E_{\eta,t}) + \Pr(\Dct > \epsilon \cap H_t^c \cap E_{\eta,t}^c) \\
		&\leq \delta/2 + 0 + \Pr(E_{\eta,t}^c) \\
		&\leq \delta/2 + \delta/2 \enspace ,
	\end{align*}
	which proves the proposition.
\end{proof}

\subsection{Ergodicity of an adaptive MCMC}	\label{subsec:ergodicity}

We can now prove the main result, which is an adaptation of  \cite[Theorem~1 and Theorem~2]{roberts2007ergodicitymcmc} to our context.

\begin{theorem}	\label{th:erg}
	Consider an adaptive MCMC algorithm defining a Markov chain $\chaint$, targeting a probability distribution $\pi$, for which conditions \ref{cond:containment_param},  \ref{cond:da_param} and \ref{cond:pergodic} hold. If furthermore, either conditions~\ref{cond:ptcontinuity} to \ref{cond:bdiscrete} or conditions \ref{cond:pcontinuity}, \ref{cond:plocalbound} and \ref{cond:bsupcond} are verified, 
		then the adaptive MCMC is ergodic relatively to $\pi$, i.e. for any initial condition $\initchain$, $\|\algt(\xt[0], \paramkt[0], \paramut[0], \cdot) - \pi\|$ converges to $0$ as $t$ goes to infinity.
	\end{theorem}
	
	\newcommand{\tXt}[1][t]{\tilde{X}_{#1}}
	\newcommand{\law}{\mathcal{L}}
	\newcommand{\Tce}{T_{C,\epsilon}}
	\begin{proof}
		By proposition~\ref{proposition:containment} and proposition~\ref{prop:dacompact} (and eventually lemma~\ref{lemma:ptc0}), conditions \ref{cond:probability_bound} and \ref{cond:diminishing_adaptation_Cx} hold.
		The rest of the proof is essentially the same as for \cite[Theorem~1]{roberts2007ergodicitymcmc}, which can be easily adapted to account for the weaker version of Diminishing Adaptation from condition~\ref{cond:diminishing_adaptation_Cx} instead of \ref{cond:diminishing_adaptation}. 
		
		The idea is, given the sequence $(\vxt)_t$ following $\algt(x_0, \theta_0, \paramut[0], \cdot)$ and $\epsilon>0$, to show the existence of $T$ such that for all $t\geq T$, another sequence $(\tXt)_t$ can be constructed, with $\|\law(X_t) - \law(\tXt)\| < \epsilon$, and $\|\law(\tXt) - \pi\| < \epsilon$ (where $\law(X)$ denotes the law of $X$). Ergodicity follows by triangular inequality. 
		
		By condition~\ref{cond:probability_bound}, there exists $T_M$ such that $\Pr(M_\epsilon(\vxt, \vparamkt) > T_M) < \epsilon$ for all $t\in\N$. Take $C := \Cx \times \Ctheta$ a compact such that $\Pr((\vxt, \vparamkt) \in C) > 1 - \epsilon/T_M$ for all $t \in \N$, and let $H_t$ denote the event $\lbrace \Dct \geq \epsilon / T_M^2 \textrm{ or } \vxt \notin \Cx \rbrace$. As $\Dct$ converges to $0$ in probability, there exists $T_D$ such that $\Pr(H_t) \leq 2\epsilon /T_M$ for all $t \geq T_D$. Let $E$ denote the event $\lbrace M_\epsilon(\vxt[T_D], \vparamkt[T_D]) < T_M \rbrace \cap \bigcap_{t = T_D}^{T_D + T_M -1} H_t^c$. This event occurs with high probability, larger than $1-3\epsilon$, and by induction and the triangular inequality, conditionally to the event $E$,  $\sup_{x \in \Cx}\|P_{\vparamkt[T_D+k+1]}(x, \cdot) -  P_{\vparamkt[T_D]}(x, \cdot)\| < \epsilon/T_M$ for all $k < T_M$. 
	
		We take $\tXt$ equal to $\vxt$ for the first $T_D$ iterations. Let $A_k$ denote the event $\lbrace\tXt[T_D+k] = \vxt[T_D+k]\rbrace$. If $E$ holds, as then $\|P_{\vparamkt[T_D+1]} (\vxt[T_D], \cdot) - P_{\vparamkt[T_D]}(\vxt[T_D],\cdot)\| < \epsilon / T_M$, we can construct $\tXt[T_D+1]$ such that $\tXt[T_D+1] \sim P_{\vparamkt[T_D]}(\vxt[T_D],\cdot)$ and that $\Pr(A_1 | E) > 1 - \epsilon/T_M$ (see e.g. \cite[Proposition~3(g)]{roberts2004general}). Inductively, we can repeat this to construct $\tXt[T_D+k]$ such that $\tXt[T_D+k] \sim P_{\vparamkt[T_D]}(\tXt[T_D+k-1],\cdot)$ and that $\Pr(A_k | E) > 1 - k\epsilon/T_M$.
		Indeed, suppose that this holds for some $k$: as $\|P_{\vparamkt[T_D+k]} (\vxt[T_D+k], \cdot) - P_{\vparamkt[T_D]}(\vxt[T_D+k],\cdot)\| < \epsilon / T_M$ on $E$, then conditionally to $A_k$ and by the same process, $\tXt[T_D+k+1]$ can be constructed such that $\tXt[T_D+k+1] \sim P_{\vparamkt[T_D]}(\tXt[T_D+k], \cdot)$ and $\Pr(A_{k+1}|E, A_k) > 1 - \epsilon / T_M$. Thus
		\begin{align*}
			\Pr(A_{k+1} | E) &= \Pr(A_{k+1} \cap A_k | E)) + \Pr(A_{k+1} \cap A_k^c | E)) \\ 
			& \geq \Pr(A_{k+1} | E, A_k))\Pr(A_k|E) + 0 \\
			& > (1 - \epsilon / T_M)(1 - k\epsilon/T_M) \\
			&> 1 - (k+1)\epsilon/T_M \enspace .
		\end{align*}
		In particular, $\Pr(A_{T_M}^c) < \epsilon + \Pr(E^c) \leq 4\epsilon$, which shows that $\|\law(\vxt[T_D+T_M]) - \law(\tXt[T_D+T_M])\| \leq 4\epsilon$.
		
		Furthermore, if $E$ holds, then $\|P_{\vparamkt[T_D]}^{T_M}(\vxt[T_D], \cdot) - \pi \| < \epsilon$. As $\tXt[T_D + T_M] \sim P_{\vparamkt[T_D]}^{T_M}(\vxt[T_D], \cdot)$ and that $\Pr(E^c) < 3 \epsilon$, by integration $\|\law(\tXt[T_D+T_M]) - \pi\| < 4\epsilon$. This concludes the proof, as this argument can be done for any $t \geq T_D$, and thus $\|\law(\vxt[t+T_M]) - \pi\| < 8\epsilon$ for all $t\geq T_D$.

	\end{proof}
	
	In the case of an adaptive Metropolis-Hastings, the following corollary can be used to prove ergodicity.
	\begin{corollary}	\label{coro:mherg}
		Consider an adaptive Metropolis-Hastings defining a Markov chain $\chaint$, with proposal kernels $(\Qp)_{\paramk \in \pkspace}$ and target distribution $\pi$, such that condition~\ref{cond:qcontinuity} is verified. If condition~\ref{cond:containment_param}, \ref{cond:da_param} and \ref{cond:pergodic} hold, then the adaptive Metropolis-Hastings is ergodic relatively to $\pi$.
	\end{corollary}
	
	\begin{proof}
		The corollary is an immediate consequence of theorem~\ref{th:erg} with lemma~\ref{lemma:mhc0}.
	\end{proof}

	\section{Applications} \label{sec:appli}

	In this section the previously derived results are applied to show the ergodicity of two different Metropolis-Hastings algorithm: the classical adaptive Metropolis-Hastings (AM) algorithm of \cite{haario2001}, and an adaptive Metropolis-Hastings inspired from the black-box optimisation algorithm CMA-ES~\cite{cmaes}. Both algorithms illustrate the use of $\vparamut$ to contain variables which influence the adaptation process and evolution of the kernel $P_{\vparamkt}(\vxt, \cdot)$, and are yet separate from both $\vxt$ and $\vparamkt$.  This allow to easily modelize the full adaptive algorithm as a time-homogeneous Markov chain, without any additional effort in proofs as no direct condition is required on $\vparamut$. In the following, the state space $\xdom$ is assumed to be $\R^\dim$, equipped with its Borel $\sigma$-algebra $\xsalgebra$ and Lebesgue measure $\xmeasure$. The density of $\pi$, $\gpi$, is also assumed continuous. The ergodicity of a non-adaptive Metropolis-Hastings is not the focus of this paper, so in the following condition~\ref{cond:pergodic} is assumed~\footnote{For example, \cite{mengersen1996geoconvmh} shows that if $\gpi$ is continuous, then the non-adaptive Metropolis-Hastings with candidate distribution $\Nl(\vxt, C)$ is $\pi$-almost everywhere ergodic for any non-degenerate covariance matrix $C$.}. Assuming that some variables (as specified further) of the two algorithms verify some bound in probability, diminishing adaptation  and thus ergodicity hold.

	\subsection{Adaptive Metropolis-Hastings}
	
	\newcommand{\mvxt}[1][t]{\bar{X}_{#1}}
	
	At time $t \in \N$, given $t_0\in \N$, a symmetric definite positive matrix $C_0 \in \cspace$ and a sequence of samples $(\vxt[k])_{k \leq t}$, AM uses a multivariate normal proposal $\Nl(\vxt, C_t)$ whose covariance matrix $C_t$ follows 
	\begin{equation} \label{eq:haario}
		\Ctt = \left\lbrace \begin{array}{lr}
			\Ct[0] & \textrm{ if }  t < t_0, \\
			s_\dim \cov\left(\vxt[0], \ldots, \vxt\right) + s_\dim \epsilon \Id  & ~~~~\textrm{ else.}
		\end{array} \right.
	\end{equation}
	The constant $s_\dim>0$ controls the scaling, and  $\cov(\vxt[0], \ldots, \vxt )$ denotes the sample covariance matrix $1/t \sum_{i=0}^{t} (\vxt[i] - \mvxt)(\vxt[i] - \mvxt)^T$, with $\bar{X}_t$ the sample mean $1/(t+1) \sum_{i=0}^t \vxt[i]$. Adding $s_n \epsilon$ times the identity matrix $\Id$ with $\epsilon > 0$ prevents the proposal from being degenerate. 
	The proposal is thus entirely parametrized by $\paramkt := \Ct$, and verifies condition~\ref{cond:qcontinuity}.
	For $t>t_0$, $\Ct$ follows the recursive formula
	\begin{equation} \label{eq:amctrec}
		\Ctt = \frac{t-1}{t}\Ct + \frac{s_\dim}{t}\left( t \mvxt[t-1]\mvxt[t-1]^T - (t+1)\mvxt\mvxt^T + \vxt\vxt^T + \epsilon \Id \right) \enspace ,
	\end{equation}
	so adding $(t, \mvxt[t-1], \mvxt)$ to $\vparamut$ suffice to have a time-homogeneous Markov chain.
	As stated in the following proposition, given that $(\vxt)_t$ is bounded enough, the update of AM is stable and the resulting algorithm is ergodic.
	\begin{proposition}	\label{pr:amerg}
		Consider the AM algorithm targeting the distribution $\pi$, generating a sequence $(\vxt, \Ct)_t$ from an initial condition $(\xt[0], C_0)$. Suppose that the density of $\pi$ is continuous, that condition~\ref{cond:pergodic} holds and that the sequence $(\vxt)_t$ is uniformly bounded in $L^2$, in the sense of the existence of $M_X > 0$ for which $\E(\|\vxt\|^2) < M_X$ for all $t \in \N$. Then AM is ergodic relatively to $\pi$.
	\end{proposition}
	
	\begin{proof}
		As previously mentioned, it is straightforward to check that condition~\ref{cond:qcontinuity} is fulfilled. It remains to verify conditions \ref{cond:containment_param} and \ref{cond:da_param} in order to deduce the proposition (using corollary~\ref{coro:mherg}). 
		
		From the bound $\E(\|\vxt\|^2) < M_X$ it easily follows that the sample mean $\bar{X}_t$ is also bounded in $L^2$ and the sample covariance $\cov(\vxt[0], \ldots, \vxt )$ in $L^\infty$, and so is $\Ct$.  
		By Markov's inequality, with high probability $\|\vxt\|_2^2$ and $\|\Ct\|_\infty$ are below some $r>0$. So with high probability $\vxt$ is contained in the compact ball $B(0,r)$, and $\Ct$ in the compact $\lbrace A \in \cspace | s_n \epsilon \Id \leq  A \leq r \Id \rbrace$ (as $s_n \epsilon \Id$ is added to $\Ct$), so  condition~\ref{cond:containment_param} is satisfied. 
		
		That diminishing adaptation holds can be inferred from Eq.~\eqref{eq:amctrec}: for any constant $\eta > 0$, by Markov's inequality
		\begin{align}	\label{eq:diffctAM}
			\Pr(\|\Ctt-\Ct\|_\infty > \eta) \leq \frac{1}{t \eta} \left( \E\left(\|\Ct\|_\infty\right) + s_\dim\left(\E(\|A_t\|_\infty) + \E(\|\vxt\vxt^T\|_\infty) + \epsilon\right) \right) \enspace,
		\end{align}
		where $A_t$ denotes the matrix $t\bar{X}_{t-1}\bar{X}_{t-1}^T - (t+1)\bar{X}_t\bar{X}_t^T$.
		The matrix $A_t$ can be developed into
		\begin{align*}
			A_t &= \frac{1}{t(t+1)}\left(\sum_{k=0}^t X_k \right) \left(\sum_{k=0}^t X_k\right)^T - \frac{1}{t}\left(\vxt \left(\sum_{k=0}^{t-1} \vxt[k]\right)^T +  \left(\sum_{k=0}^{t} \vxt[k] \right)\vxt^T \right) \enspace .
		\end{align*}
		For any vector $x$ and $y$ of $\R^n$, $\|x y^T\|_\infty \leq \|x\|_\infty \|y\|_\infty$ and thus $\|A_t\|_\infty \leq M_A := \sup_{t\geq1} 3(t+1)/t M_X^2$. Finally, denoting $M_C$ the sup of $\|\Ct\|_\infty$, from inequation~\eqref{eq:diffctAM}
		$$
		\Pr(\|\Ctt-\Ct\|_\infty > \eta) \leq \frac{1}{t\eta} \left(M_C + M_A + M_X^2 + \epsilon\right) \enspace ,
		$$
		which tends to $0$ as $t$ increases, showing that condition~\ref{cond:da_param} holds.
	\end{proof}

	\subsection{Rank-one Metropolis-Hastings Covariance Matrix Adaptation}

	The following algorithm is a truncated and simplified version of an adaptive Metropolis-Hastings algorithm which will be the subject of subsequent publication. In particular, diminishing adaptation is enforced here by directly decreasing the algorithm's learning rates, in an inefficient but simple way, as the goal here is just to illustrate the use of this paper's results. 
	
	\newcommand{\cone}{c_1}
	\newcommand{\conet}[1][t]{c_{1,#1}}
	This adaptive Metropolis-Hastings, named rank-one MH-CMA, uses a multivariate normal candidate distribution $\Qp := \Nl(\cdot, \st[]^2 \Ct[])$. The covariance matrix $\vparamkt = \st^2\Ct$ is split into a global scaling $\st$, and a symmetric positive-definite matrix $\Ct$ which governs the shape of the sampling distribution. The determinant of $\Ct$ is kept constant, so the evolution of the global scaling is entirely done by the update rule of $\st$.
	The adaptation mechanism is inspired from the Covariance Matrix Adaptation Evolution Strategy (CMA-ES), a state-of-the-art algorithm for derivative-free continuous optimization~\cite{cmaes}. 
	The adaptation of $\st$ is inspired from the so-called one-fifth success rule~\cite{schumer1968adaptive} for evolution strategies, which increases or decreases $\st$ such that about one-fifth of the generated samples are accepted. Note that a very similar adaptation mechanism has been proposed in the context of adaptive MCMC~\cite{spencer2021accelerating}. 
	Here, aiming at a acceptance rate of $\ptarget := 0.234$ as suggested in \cite{rosenthal2011optimal}, after generating a new sample $\vyt$ from $\Nl(\vxt, \st^2 \Ct)$,
	$\st$ is updated following
	\begin{equation} \label{eq:sigupdate}
		\stt = \st \exp \left( \fac_t \frac{\alpha(\vxt,\vyt)- \ptarget}{1 - \ptarget} \right) \enspace ,
	\end{equation}
	where $\fac_t > 0$ controls the amplitude of change in the scaling $\st$. At initialization, $\fac_0$ can be set to $1$.
	When $\alpha(x,y)$ is larger than $\ptarget$, the scaling is increased, which should decrease the probability of accepting the next candidate.
	The adaptation of $\Ct$ is inspired by the so-called rank-one update of CMA-ES: the successful steps are accumulated in a so-called evolution path $\pc$
	\begin{equation} \label{eq:pcupdate}
		\pcc = \left\lbrace \begin{array}{l} (1 - c_c) \pc + \sqrt{c_c ( 2 - c_c)} \frac{\vxt[t+1] - \vxt}{\st} ~~~~\textrm{if }\vyt \textrm{ is accepted,} \\ \pc ~~~~~~~~ \textrm{else.} \end{array} \right.
	\end{equation}
	A default value for $c_c$ is $\sqrt{(4 + 1/\dim) / (\dim + 4 + 2 / \dim)}$. The matrix $\Ct$ is then reinforced in the direction of the evolution path and decreased in the others, following 
	\begin{equation} \label{eq:Cupdate}
		\Ctilde = \left\lbrace \begin{array}{ll}  \left(1 - \conet \right) \Ct + \conet \pcc {\pcc}^T ~~ &\textrm{if }\vyt\textrm{ is accepted,} \\ 
			\Ct ~~~~~~~~ & \textrm{else.} \\ \end{array} \right. 
	\end{equation}
	The covariance matrix is then normalized to have constant determinant:
	\begin{equation} \label{eq:detnorm}
		\Ctt = \left(\frac{\Det\Ct[0]}{\Det(\Ctilde)} \right)^{\frac{1}{\dim}} \Ctilde \enspace . 
	\end{equation}
	The coefficient $\conet$ is in $(0,1)$ and controls the amplitude of change in $\Ctt$. At initialization, $\cone := \conet[0]$ is set to $5/3 / ((\dim + 1.3)^2 + 1)$. 
	
	Whenever a candidate $\vyt$ is accepted, the learning rates $\fac_t$ and $\conet$ are decreased by a factor $\gamma > 1$ to enforce diminishing adaptation. A variable $\tau_t$ counts the number of accepted samples at time $t$, and thus $\fac_t = \fac_0 / \gamma^{\tau_t}$ and $\conet = \cone / \gamma^{\tau_t}$. As previously mentioned, such an adaptation of the learning rates is only intended as an example, not for practice.
	
	In order to quantify the distance between two successive covariance matrices, consider $d$ the function which to two $\dim\times\dim$-symmetric positive definite matrices $A$ and $B$ associates
	\begin{equation}	\label{eq:dc}
		d(A,B) := \sqrt{\sum_{i=1}^\dim \log^2(a_i)} \enspace ,
	\end{equation}
	where $a_i$ is the $i^{\textrm{th}}$ eigenvalue of $A^{-1}B$. The function $d$ defines a distance on $\cspace$ which, among other things, is invariant under congruence transformations ($d(A,B) = d(CAC^T, CBC^T)$ for any invertible $\dim\times\dim$-matrix $C$), and makes $\cspace$ complete~\cite{forstner2003metric}. Given the update rule Eq.~\eqref{eq:Cupdate}, the following lemma characterizes the distance $d(\Ctt, \Ct)$.
	
	\begin{lemma} \label{lemma:dctctt}
		Let $(\Ct)_t$ be a sequence of symmetric definitive positive matrices following equations~\eqref{eq:Cupdate} and \eqref{eq:detnorm}. If $\vyt$ is accepted then 
		\begin{equation} \label{eq:dctctt}
			d(\Ctt, \Ct) = \sqrt{\frac{\dim-1}{\dim}} \log\left(1+ \frac{\conet \|\pcc\|^2_{\Ct}}{1-\conet} \right) \enspace .
		\end{equation}
	\end{lemma}
	
	\begin{proof}
		Let $a_i$ denote the $i^{\textrm{th}}$ eigenvalue of $\Id + \conet/(1-\conet)\Ct^{-1/2} \pcc (\Ct^{-1/2} \pcc)^T$, and $\kappa$ denote $(1-\conet)$ multiplied by the determinant normalization $(\det \Ct[0] / \det(\Ctilde))^{1/\dim}$. Using that the eigenvalues $\kappa a_i$ of $\Ct^{-1/2} \Ctt \Ct^{-1/2}$ are the same as for $\Ctt \Ct^{-1}$, $d(\Ct, \Ctt)$ is equal to the square root of $\sum_i \log^2(\kappa a_i)$.
		
		For any $v \in \R^n$ and $b\in\R$ with $\|v\|$ equal to one, the matrix $\Id + (b-1)vv^T$ has eigenvalues $(b, 1, \cdots,1)$, so  with $b$ defined as $1+ \conet/(1-\conet)\|\pcc\|_{\Ct}^2$, the eigenvalues $(a_i)_i$ are $(b, 1, \ldots, 1)$, and the determinant of $\Ctilde$ is $(1-\conet)^\dim \det(\Ct) b$. As $\det(\Ct)$ equals $\det(\Ct[0])$, $\kappa$ equals $b^{-1/\dim}$ and thus the eigenvalues $(\kappa a_i)_i$ are $(b^{1-1/\dim}, b^{-1/\dim}, \ldots, b^{-1/\dim})$. Therefore, $d(\Ctt, \Ct)$ is the square root of $(1-1/\dim)^2\log^2 b + \sum_{i=2}^\dim  (-1 / \dim)^2 \log^2 b$, which concludes the proof.
	\end{proof}
	
	The resulting algorithm forms a sequence $\chaint$, with $\vparamkt = (\st^2 \Ct)$ and $\vparamut = (\st, \Ct, \pc, \tau_t)$. 
	From Eq.~\eqref{eq:dctctt} it can be infered that if $\E(\|\pcc\|^2_{\Ct})$ is uniformly bounded and that $\conet$ converges to $0$, then diminishing adaptation holds. This is shown in the following proposition, which under further assumptions of a bound in probability on $(\vxt, \vparamkt)_t$ proves the ergodicity of the rank-one MH-CMA.
	
	\begin{proposition}
		\label{pr:mhcmaerg}
		Consider the rank-one MH-CMA algorithm targeting the distribution $\pi$. Suppose that $\pi$ admits a continuous density $\gpi$ with respect to the Lebesgue measure, and condition~\ref{cond:pergodic} that the non-adaptive kernels $P_{\sigma^2\C}$ are ergodic. If the sequence $(\vxt, \st^2\Ct, \|\pcc\|^2_{\Ct})_t$ is bounded in probability, then condition~\ref{cond:da_param}, parameter-wise diminishing adaptation, holds, and for any initial condition $\initchain$ the rank-one MH-CMA is ergodic relatively to $\pi$.
	\end{proposition}
	
	\begin{proof}
		If $\tau_t$, the number of accepted samples, converges in probability to infinity, then $\conet$ converges in probability to $0$. Therefore from lemma~\ref{lemma:dctctt}, using that $\|\pcc\|^2_{\Ct}$ is bounded in probability, $d(\Ctt, \Ct)$ also converges in probability to $0$. So condition~\ref{cond:da_param} holds, which concludes the proof by corollary~\ref{coro:mherg}.
		
		It remains to show that $\tau_t \to +\infty$ in probability, which follows from the fact that the acceptance probability can be lower bounded with high probability. 
		Indeed, by continuity of $\gpi$, the acceptance probability $a$, as defined in Eq.~\eqref{eq:praccept}, is a continuous function of $x$ and $\sigma^2 C$. So for any $\delta> 0$, as there exists a compact $K_\delta$ such that $(\vxt, \st^2\Ct)$ is in $K_\delta$ with probability at least $1-\delta$, $a(\vxt, \st^2\Ct)$ is lower bounded by some $\epsilon > 0$ whenever $(\vxt, \st^2\Ct) \in K_\delta$. Let $m_t$ be the number of visits of $K_\delta$ before time $t$, i.e. $m_t :=\sum_{k=0}^t \ind_{K_\delta}(\vxt[k], \st[k]^2\Ct[k])$. Let $A_t$ be the event for which $K_\delta$ is visited as time $t$, and $E$ be the event for which $K_\delta$ is visited infinitely often, that is $E := \lbrace w | m_t(w) \to \infty\rbrace$.
		As $E^c$ is the $\liminf_t A_t$, $\Pr(E^c) \leq \delta$, so $\Pr(\tau_t \leq M) \leq \delta + \Pr(\tau_t \leq M | E)$. 
		And in $E$, $\tau_t$ is lower bounded by a sequence of binomial random variables $(b_t)$ where $b_t$ has $m_t$ trials and $\epsilon$ success probability. As $m_t \to \infty$ on $E$, $\Pr(\tau_t \leq M | E) \leq \Pr(b_t \leq M | E) \to 0$, so for $t$ large enough $\Pr(\tau_t \leq M ) \leq 2 \delta$. This is true for any $\delta$, which concludes the proof.
	\end{proof}

\begin{appendix}

\section*{}

The following is a simple technical lemma that taking the $\sup$ over a compact preserves continuity.
\begin{lemma}	\label{lemma:supc0}
	Let $f: \xdom \times \pkspace \to \R$ be a continuous function over two locally compact metric spaces $(\xdom, \dx)$ and $(\pkspace, \dk)$, and take $\Cx$ a compact subset of $\xdom$. 
	Then the function $g : \theta \mapsto \sup_{x \in \Cx} f(x, \theta)$ is continuous.
\end{lemma}

\begin{proof}
	Clearly, the function $g$ is a lower semi-continuous function, as the $\sup$ of a continuous function. Let us prove that $g$ is also upper semi-continuous at $\theta \in \pkspace$: let $(\theta_t)_{t \in \N}$ be a sequence converging to $\theta$. By definition of the limit $\sup$, there exists a sub-sequence $(T_t)_t$ such that $g(\theta_{T_t})$ converges to $\lim \sup_{k\to \infty} g(\theta_k)$. As $\Cx$ is a compact metric space, it is complete and sequentially compact. By completeness of $\Cx$, for each $t \in \N$ there exists $x_{T_t} \in \Cx$ such that $f(x_{T_t}, \theta_{T_t}) = g(\theta_{T_t})$. As $\Cx$ is sequentially compact, a subsequence $(x_{i_t})_t$ of the sequence $(x_{T_t})_t$, converging to some $x_\infty \in \Cx$, can be extracted. By continuity of $f$, $f(x_{i_t}, \theta_{i_t})$ converges to $f(x_\infty, \theta) \leq g(\theta)$. As $f(x_{i_t}, \theta_{i_t})$ also converges to $\lim\sup g(\theta_k)$, we have $\lim\sup g(\theta_k) = f(x_\infty, \theta) \leq g(\theta)$ which shows the upper semi-continuity of $g$.
\end{proof}

\end{appendix}

\bibliography{biblio}

\end{document}